\documentclass{amsart}
\usepackage{amssymb, amsmath, mathrsfs, verbatim, url, bbm, bbold}

\makeatletter
\@namedef{subjclassname@2020}{%
  \textup{2020} Mathematics Subject Classification}
\makeatother

\theoremstyle{plain}
\newtheorem{theorem}{Theorem}[section]
\newtheorem{lemma}[theorem]{Lemma}

\newtheorem{corollary}[theorem]{Corollary}

\theoremstyle{definition}

\newtheorem{question}[theorem]{Question}
\newtheorem{conjecture}[theorem]{Conjecture}

\newcommand{\re}{\upharpoonright}
\newcommand{\id}{\mathsf{id}}
\newcommand{\Gd}{\mathsf{G_\delta}}
\newcommand{\CDH}{\mathsf{CDH}}
\newcommand{\ZFC}{\mathsf{ZFC}}
\newcommand{\BB}{\mathcal{B}}
\newcommand{\RRR}{\mathbb{R}}
\newcommand{\ZZZ}{\mathbb{Z}}
\newcommand{\cccc}{\mathfrak{c}}
\newcommand{\bbbb}{\mathfrak{b}}

\begin{document}

\title{Countable dense homogeneity and topological groups}

\dedicatory{Dedicated to the memory of Peter Nyikos}

\author{Claudio Agostini}
\address{Institut f\"{u}r Diskrete Mathematik und Geometrie
\newline\indent Technische Universit\"{a}t Wien
\newline\indent  Wiedner Hauptstra\ss e 8--10/104
\newline\indent 1040 Vienna, Austria}
\email{claudio.agostini@tuwien.ac.at}
\urladdr{https://sites.google.com/view/claudioagostiniswebsite}

\author{Andrea Medini}
\address{Institut f\"{u}r Diskrete Mathematik und Geometrie
\newline\indent Technische Universit\"{a}t Wien
\newline\indent  Wiedner Hauptstra\ss e 8--10/104
\newline\indent 1040 Vienna, Austria}
\email{andrea.medini@tuwien.ac.at}
\urladdr{https://www.dmg.tuwien.ac.at/medini/}

\author{Lyubomyr Zdomskyy}
\address{Institut f\"{u}r Diskrete Mathematik und Geometrie
\newline\indent Technische Universit\"{a}t Wien
\newline\indent  Wiedner Hauptstra\ss e 8--10/104
\newline\indent 1040 Vienna, Austria}
\email{lyubomyr.zdomskyy@tuwien.ac.at}
\urladdr{https://www.dmg.tuwien.ac.at/zdomskyy/}

\date{January 16, 2025}

\thanks{This research was funded in whole by the Austrian Science Fund (FWF) DOI 10.55776/P35655 (Agostini), 10.55776/P35588 (Medini) and 10.55776/I5930 (Zdomskyy). For open access purposes, the authors have applied a CC BY public copyright license to any authors-accepted manuscript version arising from this submission.}

\begin{abstract}
Building on results of Medvedev, we construct a $\ZFC$ example of a non-Polish topological group that is countable dense homogeneous. Our example is a dense subgroup of $\ZZZ^\omega$ of size $\bbbb$ that is a $\lambda$-set. We also conjecture that every countable dense homogenous Baire topological group with no isolated points contains a copy of the Cantor set, and give a proof in a very special case.
\end{abstract}

\subjclass[2020]{Primary 54H11, 54G20.}

\keywords{Countable dense homogeneous, topological group, $\lambda$-set, Baire space, h-homogeneous, homogeneous.}

\maketitle

\tableofcontents

\section{Introduction}\label{section_introduction}

As is common in the literature about countable dense homogeneity, all spaces (including groups) are assumed to be separable and metrizable. By \emph{group} we mean topological group. By \emph{countable} we mean at most countable. Our reference for general topology is \cite{engelking}. Our reference for groups is \cite{arhangelskii_tkachenko}. Our reference for descriptive set theory is \cite{kechris}. For all other set-theoretic notions, we refer to \cite{kunen}. A space $X$ is \emph{countable dense homogeneous} ($\CDH$) if for every pair $(D,E)$ of countable dense subsets of $X$ there exists a homeomorphism $h:X\longrightarrow X$ such that $h[D]=E$.

The fundamental positive result in the theory of $\CDH$ spaces is due to Anderson, Curtis and van Mill \cite[Theorem 5.2]{anderson_curtis_van_mill}, and it states that every strongly locally homogeneous Polish space is $\CDH$.\footnote{\,Recall that a space $X$ is \emph{strongly locally homogeneous} if there exists a base $\BB$ for $X$ such that for every $U\in\BB$ and $(x,y)\in U\times U$ there exists a homeomorphism $h:X\longrightarrow X$ such that $h(x)=y$ and $h(z)=z$ for every $z\in X\setminus U$.} In particular, familiar spaces like the Cantor set $2^\omega$, the Baire space $\omega^\omega$, the Euclidean spaces $\RRR^n$, the spheres $\mathsf{S}^n$ and the Hilbert cube $[0,1]^\omega$ are all examples of $\CDH$ spaces. See \cite[Sections 14-16]{arhangelskii_van_mill} for much more on this topic.

Fitzpatrick and Zhou \cite[Problem 6]{fitzpatrick_zhou_densely} asked for non-Polish examples of $\CDH$ spaces (this is Problem 389 from the book ``Open problems in topology''), and the present article belongs to this line of research. Consistent examples were constructed in \cite[Theorem 4.6]{fitzpatrick_zhou_densely} and \cite[Theorem 3.5]{baldwin_beaudoin}. The first $\ZFC$ example was given in \cite{farah_hrusak_martinez_ranero}, using metamathematical methods. Another breakthrough was obtained in \cite{hernandez_gutierrez_hrusak_van_mill}, where the authors gave a more direct, ``down-to-earth'' $\ZFC$ construction. This approach was further developed in \cite{medini_cdh_powers}, \cite{medvedev_dh} and \cite{medvedev_dense}. In particular, one of these results of Medvedev (namely, Theorem \ref{theorem_medvedev_lambda}) will be a fundamental ingredient in obtaining our main result.

A more recent source of inspiration was the following interesting open question \cite[Question 1.2]{dobrowolski_krupski_marciszewski}. We remark that Hern\'{a}ndez-Guti\'{e}rrez \cite[Corollary 3.8]{hernandez_gutierrez} showed that, consistently, there exists a non-Polish $\mathsf{C}_p$-space that is $\CDH$.
\begin{question}[Dobrowolski, Krupski, Marciszewski]\label{question_vector_space}
Is there a $\ZFC$ example of a non-Polish $\CDH$ topological vector space?\footnote{\,Although the authors are not explicit about it, the remark that precedes \cite[Corollary 3.4]{dobrowolski_krupski_marciszewski} suggests that they are only considering topological vector spaces over $\RRR$.}
\end{question}

In fact, in his FWF research project P 35655, the second-listed author pointed out that the following (easier) question is also open.
\begin{question}[Medini]\label{question_group}
Is there a $\ZFC$ example of a non-Polish $\CDH$ group?
\end{question}

The main result of this article shows that Question \ref{question_group} has an affirmative answer (see Corollary \ref{corollary_main}). We remark that, for a consistent example of a non-Polish $\CDH$ group, it is not necessary to resort to $\mathsf{C}_p$-spaces. The earliest example seems to be \cite[Theorem 21]{medini_milovich}. In fact, every non-meager $\mathsf{P}$-filter, viewed as a subspace of $2^\omega$, gives such a group by \cite[Theorem 1.6]{hernandez_gutierrez_hrusak} (see also \cite[Theorem 10]{kunen_medini_zdomskyy}). However, the existence of a non-meager $\mathsf{P}$-filter in $\ZFC$ is a long-standing open problem \cite[Question 0.1]{just_mathias_prikry_simon}.

\section{More preliminaries and terminology}\label{section_preliminaries}

Given a space $Z$, we will say that a subspace $X$ of $Z$ is a \emph{copy} of a space $Y$ if $X$ is homeomorphic to $Y$. A space is \emph{crowded} if it is non-empty and it has no isolated points. A subset $S$ of a space $X$ is \emph{meager} in $X$ if there exist closed nowhere dense subsets $C_n$ of $X$ for $n\in\omega$ such that $S\subseteq\bigcup_{n\in\omega}C_n$. A space $X$ is \emph{meager} if $X$ is a meager subset of $X$. A space $X$ is \emph{Baire} if no non-empty open subset of $X$ is meager in $X$.

A space $X$ is \emph{homogeneous} if for every $(x,y)\in X\times X$ there exists a homeomorphism $h:X\longrightarrow X$ such that $h(x)=y$. Using translations, one sees that every group is homogeneous. A space $X$ is \emph{h-homogeneous} if every non-empty clopen subspace of $X$ is homeomorphic to $X$.

A subset $S$ of a group $X$ is \emph{precompact} if for every neighborhood $U$ of the identity in $X$ there exists a finite $F\subseteq X$ such that $S\subseteq FU\cap UF$. A group $X$ is \emph{locally precompact} if there exists a precompact neighborhood of the identity in $X$. The fact that $\ZZZ^\omega$ is not locally precompact (which is a simple exercise) will be crucial in the proof of Corollary \ref{corollary_main}. Here, we use $\ZZZ$ to denote the integers with the discrete topology and the standard operation of addition. Recall that every group $X$ is a dense subgroup of a Polish group $\varrho X$, known as the \emph{Ra\u{\i}kov completion} of $X$ (see \cite[Theorem 3.6.10, Proposition 3.6.20 and Proposition 4.3.8]{arhangelskii_tkachenko}).

Given $f,g\in\ZZZ^\omega$, we will write $f\leq^\ast g$ (respectively $f<^\ast g$ or $f=^\ast g$) to mean that $f(n)\leq g(n)$ (respectively $f(n)<g(n)$ or $f(n)=g(n)$) for all but finitely many $n\in\omega$. We will say that $S\subseteq\omega^\omega$ is \emph{bounded} if there exists $g\in\omega^\omega$ such that $f\leq^\ast g$ for every $f\in S$. Recall that
$$
\bbbb=\mathsf{min}\{|S|:S\text{ is an unbounded subset of }\omega^\omega\}.
$$

A space $X$ is a \emph{$\lambda$-set} if every countable subset of $X$ is a $\Gd$-subset of $X$. It is easy to realize that every crowded $\lambda$-set is meager, hence non-Polish. The only other fact concerning $\lambda$-sets that we will need is the following classical result \cite[Th\'{e}or\`{e}me 1]{rothberger}. See \cite[Section 2]{hernandez_gutierrez_hrusak_van_mill} and \cite[Section 5]{miller} for much more on this topic.

\begin{lemma}[Rothberger]\label{lemma_less_than_b}
Let $X$ be a space such that $|X|<\bbbb$. Then $X$ is a $\lambda$-set.
\end{lemma}
A subspace $X$ of $2^\omega$ is a \emph{Bernstein set} if $X\cap K\neq\varnothing$ and $(2^\omega\setminus X)\cap K\neq\varnothing$ for every copy $K$ of $2^\omega$ in $2^\omega$. It is easy to see that Bernstein sets exist in $\ZFC$, and that they are Baire spaces (in fact, using \cite[Corollary 1.9.13]{van_mill}, one can show that every closed subspace of a Bernstein set is a Baire space).

\section{The main result}\label{section_main}

In this section, we will obtain our main result (see Corollary \ref{corollary_main}), although Theorem \ref{theorem_main} seems to be of independent interest. We begin by stating two theorems of Medvedev that will be our fundamental tools.

The following result shows that, in the context of zero-dimensional groups, there is a particularly simple sufficient condition for h-homogeneity. For a proof, see \cite[Theorem 8]{medvedev_homogeneity}.\footnote{\,We remark that in \cite{medvedev_homogeneity} the assumption ``zero-dimensional paracompact'' often appears, while the stronger assumption ``ultraparacompact'' is in fact used. (Recall that a space $X$ is \emph{ultraparacompact} if every open cover of $X$ can be refined by a  clopen partition of $X$.) This does not affect us, since zero-dimensionality easily implies ultraparacompactness in the separable metrizable realm.}
\begin{theorem}[Medvedev]\label{theorem_medvedev_group}
Let $X$ be a zero-dimensional group. If $X$ is not locally precompact then $X$ is h-homogeneous.
\end{theorem}

The following is \cite[Corollary 6]{medvedev_dh}, and it can be viewed as the culmination of the technique introduced in \cite{hernandez_gutierrez_hrusak_van_mill}. The fact that every meager $\CDH$ space is a $\lambda$-set is due to Fitzpatrick and Zhou \cite[Theorem 3.4]{fitzpatrick_zhou_baire}.
\begin{theorem}[Medvedev]\label{theorem_medvedev_lambda}
Let $X$ be an uncountable zero-dimensional h-homogeneous space. Then the following conditions are equivalent:
\begin{itemize}
\item $X$ is a meager $\CDH$ space,
\item $X$ is a $\lambda$-set.
\end{itemize}
\end{theorem}

\begin{theorem}\label{theorem_main}
There exists a dense subgroup of $\ZZZ^\omega$ of size $\bbbb$ that is a $\lambda$-set.
\end{theorem}

\begin{proof}
Given $f\in\ZZZ^\omega$, we will denote by $|f|$ the function defined by $|f|(n)=|f(n)|$ for $n\in\omega$. We will denote by $\vec{0}\in\ZZZ^\omega$ the constant function with value $0$. Given $S\subseteq\ZZZ^\omega$, we will denote by $\langle S\rangle$ the subgroup of $\ZZZ^\omega$ generated by $S$.

Using transfinite recursion, it is straightforward to construct $f_\alpha\in\ZZZ^\omega$ for $\alpha<\bbbb$ such that the following conditions are satisfied:
\begin{enumerate}
\item $f_\alpha(n)>0$ for each $\alpha<\bbbb$ and $n<\omega$,
\item $f_\alpha>^\ast k_0\cdot f_{\beta_0}+\cdots+k_n\cdot f_{\beta_n}$ whenever $n<\omega$, $k_0,\ldots,k_n\in\omega$, and $\beta_0,\ldots,\beta_n<\alpha<\bbbb$.
\end{enumerate}
Observe that $f_\alpha\neq f_\beta$ whenever $\alpha\neq\beta$. Set $Q=\{z\in\ZZZ^\omega:z=^\ast\vec{0}\}$ and
$$
X=\langle\{f_\alpha:\alpha<\bbbb\}\cup Q\rangle.
$$
It is clear that $X$ is dense in $\ZZZ^\omega$ and that $|X|=\bbbb$.

It remains to show that $X$ is a $\lambda$-set. For any given $z\in X\setminus Q$, fix $k(z)\in\ZZZ\setminus\{0\}$ and $\alpha(z)<\bbbb$ such that
$$
z=^\ast k(z)\cdot f_{\alpha(z)}+z'
$$
for some $z'\in\langle\{f_\beta:\beta<\alpha(z)\}\rangle$. Now pick a countable $C\subseteq X$. Since $\bbbb$ has uncountable cofinality, it is possible to fix $\delta<\bbbb$ such that $\delta>\mathsf{sup}\{\alpha(z):z\in C\}$.

\noindent\textbf{Claim.} Let $x\in C$, and let $y\in X\setminus Q$. Assume that $\alpha(y)>\delta$. Then
$$
|x|\leq^\ast f_\delta <^\ast |y|.
$$
\noindent\textit{Proof.} The first inequality follows easily from our choice of $\delta$, so we will only prove the second one. Pick $n\in\omega$, $k_0,\ldots,k_n=k(y)\in\ZZZ$ and $\alpha_0<\cdots <\alpha_n=\alpha(y)<\bbbb$ such that
$$
y=^\ast k_0\cdot f_{\alpha_0}+\cdots +k_n\cdot f_{\alpha_n}.
$$
Set
$$
y'=f_{\alpha(y)}-\sum_{i<n}|k_i|\cdot f_{\alpha_i},
$$
and observe that $y'\geq^\ast\vec{0}$ by condition $(2)$.

Next, we will prove that $|y|\geq^\ast y'$. First assume that $k(b)>0$. In this case, it is easy to realize that $y\geq^\ast y'\geq^\ast\vec{0}$, hence $|y|=^\ast y \geq^\ast y'$. Now assume that $k(b)<0$. In this case, one sees that
$$
-y=^\ast (-k(y))\cdot f_{\alpha(y)}-\sum_{i<n}k_i\cdot f_{\alpha_i}\geq f_{\alpha(y)}-\sum_{i<n}|k_i|\cdot f_{\alpha_i}=y'\geq^\ast\vec{0},
$$
so $|y|=^\ast -y \geq^\ast y'$. In conclusion, we must have
$$
|y|\geq^\ast y'= f_{\alpha(y)}-\sum_{i<n}|k_i|\cdot f_{\alpha_i}>^\ast f_\delta,
$$
where the last inequality follows from the assumption that $\alpha(y)\geq\delta$. $\blacksquare$

Finally, define
$$
G=\{z\in\ZZZ^\omega:|z(n)|\leq f_\delta(n)\text{ for infinitely many }n\in\ZZZ\},
$$
and observe that $G$ is a $\Gd$-subset of $\ZZZ^\omega$. Furthermore, as a consequence of the above Claim, it is easy to realize that $C\subseteq X\cap G$ and $|X\cap G|<\bbbb$. Therefore, by Lemma \ref{lemma_less_than_b}, there exists a $\Gd$-subset $G'$ of $\ZZZ^\omega$ such that $X\cap G\cap G'=C$. This clearly shows that $C$ is a $\Gd$-subset of $X$, as desired.
\end{proof}

We remark that the above theorem (to the best of our knowledge) gives the first $\ZFC$ example of an uncountable group that is a $\lambda$-set.\footnote{\,Rather unsurprisingly, we suggest the terminology \emph{$\lambda$-group}.} (For a consistent example, under the assumption that $\omega_1<\bbbb$, consider an arbitrary group of size $\omega_1$ and apply Lemma \ref{lemma_less_than_b}.)

\begin{corollary}\label{corollary_main}
There exists a non-Polish $\CDH$ group.	
\end{corollary}
\begin{proof}
Let $X$	be the group given by Theorem \ref{theorem_main}. By Theorems \ref{theorem_medvedev_group} and \ref{theorem_medvedev_lambda}, in order to conclude the proof, it will be enough to show that $X$ is not locally precompact. Assume, in order to get a contraction, that $X$ is locally precompact. Using \cite[Lemma 3.7.5]{arhangelskii_tkachenko}, one can easily show that every group containing a locally precompact dense subgroup is locally precompact. Since $\ZZZ^\omega$ is not locally precompact, we have reached a contradiction.
\end{proof}

Given that our example is meager, the following question is still open (as we mentioned in Section \ref{section_introduction}, non-meager $\mathsf{P}$-filters give consistent examples).
\begin{question}\label{question_baire_group}
Is there a $\ZFC$ example of a non-Polish Baire $\CDH$ group?
\end{question}

Since every $\CDH$ topological vector space is Baire (see the remarks that follow Conjecture \ref{conjecture_baire_groups}), an affirmative answer to Question \ref{question_baire_group} could be viewed as a further ``baby-step'' towards answering Question \ref{question_vector_space}.

\section{A conjecture about Baire $\CDH$ groups}\label{section_conjecture}

This section is motivated by the following open problem \cite[Question 4.7]{hernandez_gutierrez_hrusak_van_mill}.
\begin{question}[Hern\'{a}ndez-Guti\'{e}rrez, Hru\v{s}\'{a}k, van Mill]\label{question_baire}
Is it consistent with $\ZFC$ to have a crowded Baire $\CDH$ space of size less than $\cccc$?
\end{question}

Assuming Martin's Axiom for $\sigma$-centered posets, Baldwin and Beaudoin \cite[Theorem 3.5]{baldwin_beaudoin} constructed a $\CDH$ Bernstein set in $2^\omega$ (see also \cite{medini_cdh_products}). In an attempt to better understand $\CDH$ groups, we tried to modify this construction in order to obtain a subgroup of $2^\omega$. However, we ran into serious difficulties. This lead us to the following conjecture, which would resolve Question \ref{question_baire} in the realm of groups.
\begin{conjecture}\label{conjecture_baire_groups}
Every crowded Baire $\CDH$ group contains a copy of $2^\omega$.
\end{conjecture}

Another interesting consequence of Conjecture \ref{conjecture_baire_groups} would be the equivalence of the following conditions for every crowded $\CDH$ group $X$:
\begin{itemize}
\item $X$ is Baire,
\item $X$ contains a copy of $2^\omega$.	
\end{itemize}
In fact, Dobrowolski, Krupski and Marciszewski \cite[Theorem 3.1]{dobrowolski_krupski_marciszewski} proved that every homogeneous $\CDH$ space containing a copy of $2^\omega$ is Baire. Notice however that mere homogeneity would not be sufficient for this equivalence, since the $\CDH$ Bernstein set mentioned above can be easily made homogeneous as well.

We conclude this article by proving Conjecture \ref{conjecture_baire_groups} in the very special case of groups that have index $2$ in their completion (see Theorem \ref{theorem_conjecture}). In the following results, although the group $X$ is not assumed to be commutative, we will still use the additive notation for the sake of readability. In particular, given $f,g:X\longrightarrow X$, we will use $f-g$ to denote the function obtained by setting $(f-g)(x)=f(x)-g(x)$. We will use $\id_X$ to denote the identity function on a set $X$.
\begin{lemma}\label{lemma_conjecture}
Let $X$ be a crowded Baire $\CDH$ group. Then there exists a homeomorphism $h:X\longrightarrow X$ such that $(h-\id_X)[X]$ is uncountable.
\end{lemma}
\begin{proof}
Fix a countable dense subgroup $D$ of $X$. Notice that $X$ must be uncountable, because it is a crowded $\CDH$ space. Therefore, we can fix $x_0\in X\setminus D$. Set $E=x_0+D$, and observe that $D\cap E=\varnothing$. Let $h:X\longrightarrow X$ be a homeomorphism such that $h[D\cup E]=D$. Set $f=h-\id_X$.

Assume, in order to get a contradiction, that $f[X]$ is countable. Then, since $X$ is Baire and
$$
X=\bigcup_{y\in f[X]}f^{-1}(y),
$$
there must be a non-empty open subset $U$ of $X$ such that $f\re U$ is constant. In other words, there exists $x_1\in X$ such that $h(z)=x_1+z$ for every $z\in U$. Fix $x\in D\cap U$ and $y\in E\cap U$. Notice that $x_1\in D$ because $x_1+x=h(x)\in D$ and $D$ is a subgroup of $X$. Since $x_1+y=h(y)\in D$, it follows that $y\in D$, a contradiction.
\end{proof}

\begin{theorem}\label{theorem_conjecture}
Let $X$ be a crowded Baire $\CDH$ group. Assume that $X$ has index $2$ in its Ra\u{\i}kov completion $\varrho X$. Then $X$ contains a copy of $2^\omega$.
\end{theorem}
\begin{proof}
Let $h$ denote the homeomorphism given by Lemma \ref{lemma_conjecture}. By \cite[Exercise 3.10]{kechris}, we can fix a $\Gd$-subset $G$ of $\varrho X$ and a homeomorphism $\widetilde{h}:G\longrightarrow G$ such that $h\subseteq\widetilde{h}$. Notice that $G$ is a Polish space by \cite[Theorem 3.11]{kechris}, hence $(\widetilde{h}-\id_G)[G]$ is analytic. Furthermore, Lemma \ref{lemma_conjecture} guarantees that $(\widetilde{h}-\id_G)[G]$ is uncountable. Since analytic sets have the perfect set property (see \cite[Theorem 29.1]{kechris}), in order to conclude the proof, it will be enough to show that $(\widetilde{h}-\id_G)[G]\subseteq X$.

Since $X$ has index $2$ in $\varrho X$, there exists $z_0\in\varrho X$ such that $z_0+X=\varrho X\setminus X$. Pick $z\in G$. If $z\in X$ then clearly $(\widetilde{h}-\id_G)(z)=h(z)-z\in X$, so assume that $z\in G\setminus X$. Then $z=z_0+x$ and $\widetilde{h}(z)=z_0+y$ for suitable $x,y\in X$. Hence
$$
(\widetilde{h}-\id_G)(z)=(z_0+y)-(z_0+x)=z_0+(y-x)-z_0\in z_0+X-z_0=X,
$$
as desired, where the last equality follows from the well-known (and easy to prove) fact that subgroups of index two are normal.
\end{proof}

\end{document}